\documentclass[12pt]{article}
\usepackage[centertags]{amsmath}
\usepackage{amsfonts}
\usepackage{amssymb}
\usepackage{amsthm}
\usepackage{newlfont}
\usepackage{graphicx}

\newfont{\bb}{msbm10 at 12pt}
\def\r{\hbox{\bb R}}

\def\e{\hbox{\bf E}}
\def\t{\hbox{\bf T}}
\def\n{\hbox{\bf N}}
\def\b{\hbox{\bf B}}

\newtheorem{theorem}{Theorem}[section]
\newtheorem{definition}[theorem]{Definition}

\begin{document}

\title{Slant helices in Euclidean 4-space $\e^4$}
\author{ Ahmad T. Ali\\Mathematics Department\\
 Faculty of Science, Al-Azhar University\\
 Nasr City, 11448, Cairo, Egypt\\
email: atali71@yahoo.com\\
\vspace*{1cm}\\
 Rafael L\'opez\footnote{Partially
supported by MEC-FEDER
 grant no. MTM2007-61775.}\\
Departamento de Geometr\'{\i}a y Topolog\'{\i}a\\
Universidad de Granada\\
18071 Granada, Spain\\
email: rcamino@ugr.es}
\date{}

\maketitle
\begin{abstract} We consider a unit speed curve $\alpha$ in Euclidean four-dimensional space $\e^4$ and denote the Frenet frame by $\{\t,\n,\b_1,\b_2\}$. We say that $\alpha$ is a slant helix if its principal normal vector $\n$ makes a constant angle  with a fixed direction  $U$. In this work we give different characterizations of such curves in terms of their curvatures.

\end{abstract}



%%%%%%%%%%%%%%%%%%%%%%%%%%%%%%%%%%%%%%%%%%%%%%%
\section{Introduction and statement of results}
%%%%%%%%%%%%%%%%%%%%%%%%%%%%%%%%%%%%%%%%%%%%%%%
A helix in Euclidean 3-space $\e^3$ is a curve whose tangent lines make a constant angle with a fixed direction. A helix curve is characterized by the fact that the ratio $\tau/\kappa$ is constant along the curve, where $\tau$ and $\kappa$ denote the torsion and the curvature, respectively. Helices are well known curves in classical differential geometry of space curves \cite{mp} and we refer to the reader for recent
works on this type of curves \cite{gl1,sc}.
Recently, Izumiya and Takeuchi have introduced the concept of slant helix by saying that the normal lines make a constant angle with a fixed direction \cite{it1}. They characterize a slant helix if and only if the function
$$\dfrac{\kappa^2}{(\kappa^2+\tau^2)^{3/2}}\Big(\dfrac{\tau}{\kappa}\Big)'$$
is constant. This article  motivated generalizations in a twofold sense: first, by increasing the dimension of Euclidean space \cite{ky,okkk}; second, by considering analogous problems in other ambient spaces, mainly, in  Minkowski space $\e_1^n$ \cite{al1,al2,ba,ey1,ko,ps}.

In this work we consider the generalization of the concept of slant helix in Euclidean 4-space $\e^4$. Let $\alpha:I\subset\r\rightarrow\e^4$ be an arbitrary curve in $\e^4$. Recall that the curve $\alpha$ is said to be of unit speed (or parameterized by arclength function $s$) if $\langle\alpha'(s),\alpha'(s)\rangle=1$, where $\langle,\rangle$ is the standard scalar product in the Euclidean space $\e^4$ given by
$$\langle X,Y\rangle=x_1y_1+x_2y_2+x_3y_3+x_4y_4,$$
for each $X=(x_1,x_2,x_3,x_4)$, $Y=(y_1,y_2,y_3,y_4)\in\e^4$.

Let $\{\t(s),\n(s),\b_1(s),\b_2(s)\}$ be the moving frame along $\alpha$, where $\t,\n,\b_1$ and $\b_2$ denote the tangent, the principal normal, the first binormal and second binormal vector fields, respectively. Here $\t(s)$, $\n(s)$, $\b_1(s)$ and $\b_2(s)$ are mutually orthogonal vectors satisfying
$$\langle\t,\t\rangle=\langle\n,\n\rangle=\langle\b_1,\b_1\rangle=\langle\b_2,\b_2\rangle=1.$$
The Frenet equations for $\alpha$ are given by
\begin{equation}\label{u2}
 \left[
   \begin{array}{c}
     \t' \\
     \n' \\
     \b_1' \\
     \b_2'\\
   \end{array}
 \right]=\left[
           \begin{array}{cccc}
             0 & \kappa_1 & 0 & 0 \\
             -\kappa_1 & 0 & \kappa_2 & 0 \\
             0 & -\kappa_2 & 0 & \kappa_3 \\
             0 & 0 & -\kappa_3 & 0 \\
           \end{array}
         \right]\left[
   \begin{array}{c}
     \t \\
     \n \\
     \b_1 \\
     \b_2\\
   \end{array}
 \right].
 \end{equation}

Recall the functions $\kappa_1(s)$, $\kappa_2(s)$ and $\kappa_3(s)$ are called respectively, the first, the second and the third curvatures of $\alpha$. If $\kappa_3(s)=0$ for any $s\in I$, then $\b_2(s)$ is a constant vector $B$ and the curve $\alpha$ lies in a three-dimensional affine subespace orthogonal to $B$, which is isometric to the Euclidean 3-space $\e^3$.

We will assume throughout this work that all the three curvatures satisfy $\kappa_i(s)\not=0$ for any $s\in I$, $1\leq i\leq 3$.

\begin{definition} \label{df-1} A unit speed curve $\alpha:I\rightarrow\e^4$ is said to be a generalized  helix if there exists a non-zero constant vector field $U$ and a vector field $X\in\{\t,\n,\b_1,\b_2\}$ such that the function
$$s\longmapsto\langle X(s),U\rangle,\hspace*{.5cm}s\in I$$
 is constant.
\end{definition}

Among the possibilities to choose the vector field $X$ we have:

\begin{enumerate}
\item If $X$ is the unit tangent vector field $\t$, $\alpha$ is called a cylindrical helix. It is known that $\alpha(s)$ is a  cylindrical helix if and only if the function
$$
\dfrac{\kappa_1^2}{\kappa_2^2}+\Big[\dfrac{1}{\kappa_3}\Big(\dfrac{\kappa_1}{\kappa_2}\Big)^{\prime}\Big]^2$$
is constant. See \cite{ma, oh}.
\item If $X$ is the vector field  $\b_2$, then the curve is called a $\b_2$-slant curve. Moreover $\alpha$ is a such curve if and only if
the function
$$
\dfrac{\kappa_3^2}{\kappa_2^2}+\Big[\dfrac{1}{\kappa_1}\Big(\dfrac{\kappa_3}{\kappa_2}\Big)^{\prime}\Big]^2$$
is constant. See \cite{okkk}.
\end{enumerate}

\begin{definition} A unit speed curve $\alpha:I\rightarrow\e^4$ is called slant helix if its unit principal normal vector $\n$ makes a constant angle  with a fixed direction  $U$.
\end{definition}

Our main result in this work is the following characterization of slant helices.

\begin{theorem}\label{th-main} Let $\alpha:I\rightarrow\e^4$ be a unit speed curve in $\e^4$. Then $\alpha$  is a slant helix if and only if
the function
$$
\Big(\int\kappa_1(s)ds\Big)^2
+\Big[\dfrac{1}{\kappa_3}\Big(\dfrac{\kappa_1}{\kappa_2}\int\kappa_1ds\Big)^{\prime}+\dfrac{\kappa_2}{\kappa_3}\Big]^2
+
\Big(\dfrac{\kappa_1}{\kappa_2}\int\kappa_1ds\Big)^2$$
is constant. Moreover, this constant agrees with $\tan^2\theta$, being $\theta$ the angle that makes $\n$ with the fixed direction $U$ that determines $\alpha$.
\end{theorem}

%%%%%%%%%%%%%%%%%%%%%%%%%%%%%%%%%%%%%%%%%%%
\section{Proof of Theorem \ref{th-main}}
%%%%%%%%%%%%%%%%%%%%%%%%%%%%%%%%%%%%%%%%%%%

Let  $\alpha$ be a unit speed curve in $\e^4$. Assume that $\alpha$ is a slant curve. Let $U$ be the direction with which $\n$ makes a constant angle $\theta$ (suppose that $\langle U,U\rangle=1$).  Consider the differentiable functions $a_i$, $1\leq i\leq 4$,
\begin{equation}\label{u3}
U=a_1(s)\t(s)+a_2(s) \n(s)+a_3(s) \b_1(s)+a_4(s)\b_2(s),\ \ s\in I,
\end{equation}
that is,
$$a_1=\langle\t,U\rangle, \ a_2=\langle\n,U\rangle,\ a_3=\langle\b_1,U\rangle,\ a_4=\langle\b_2,U\rangle.$$
Because the vector field $U$ is constant, a differentiation in (\ref{u3}) together (\ref{u2}) gives the following ordinary differential equation system
\begin{equation}\label{u4}
\left.\begin{array}{ll}
a_1'-\kappa_1 a_2&=0\\
a_2'+\kappa_1 a_1-\kappa_2 a_3&=0\\
a_3'+\kappa_2 a_2-\kappa_3 a_4&=0\\
a_4'+\kappa_3 a_3 &=0
\end{array}\right\}
\end{equation}
Then the function  $a_2(s)=\langle \n(s),U\rangle$ is constant, and it agrees with $\cos\theta$. Then (\ref{u4}) gives
\begin{equation}\label{u5}
\left.\begin{array}{ll}
a_1'-\kappa_1 a_2&=0\\
\kappa_1 a_1-\kappa_2 a_3&=0\\
a_3'+\kappa_2 a_2-\kappa_3 a_4&=0\\
a_4'+\kappa_3 a_3 &=0
\end{array}\right\}
\end{equation}
The first third equations in (\ref{u5}) lead to
\begin{equation}\label{u6}
\left.\begin{array}{ll}
a_1&=a_2\int\kappa_1ds\\
a_3&=a_2\dfrac{\kappa_1}{\kappa_2}\int\kappa_1ds\\
a_4&=a_2\Big[\dfrac{1}{\kappa_3}\Big(\dfrac{\kappa_1}{\kappa_2}\int\kappa_1ds\Big)^{\prime}+\dfrac{\kappa_2}{\kappa_3}\Big]
\end{array}\right\}
\end{equation}
We do the change of variables:
$$t(s)=\int^s\kappa_3(u) du,\hspace*{.5cm}\frac{dt}{ds}=\kappa_3(s).$$
In particular, and from  (\ref{u5}), we have
$$a_3'(t)=a_4-a_2\frac{\kappa_2}{\kappa_3}.$$
As a consequence, if $\alpha$ is a slant helix, the last equation of (\ref{u5}) yields
\begin{equation}\label{u8}
a_4''(t)+a_4(t)-a_2\frac{\kappa_2(t)}{\kappa_3(t)}=0.
\end{equation}
The general solution of this equation is
\begin{equation}\label{u9}
a_4(t)=a_2\Bigg[\Big(A-\int\dfrac{\kappa_2(t)}{\kappa_3(t)}\sin{t}\,dt\Big)\cos{t}+
\Big(B+\int\dfrac{\kappa_2(t)}{\kappa_3(t)}\cos{t}\,dt\Big)\sin{t}\Bigg],
\end{equation}
where $A$ and $B$ are arbitrary constants. Then (\ref{u9}) takes the following form
\begin{equation}\label{u10}
\begin{array}{ll}
a_4(s)=&a_2\Big[\Big(A-\int\Big[\kappa_2(s)\sin{\int\kappa_3(s)ds}\Big]ds\Big)\cos{\int\kappa_3(s)ds}\\
&+
\Big(B+\int\Big[\kappa_2(s)\cos{\int\kappa_3(s)ds}\Big]ds\Big)\sin{\int\kappa_3(s)ds}\Big].
\end{array}
\end{equation}
From (\ref{u5}), the function $a_3$ is given by
\begin{equation}\label{u11}
\begin{array}{ll}
a_3(s)=&a_2\Big[\Big(A-\int\Big[\kappa_2(s)\sin{\int\kappa_3(s)ds}\Big]ds\Big)\sin{\int\kappa_3(s)ds}\\
&-\Big(B+\int\Big[\kappa_2(s)\cos{\int\kappa_3(s)ds}\Big]ds\Big)\cos{\int\kappa_3(s)ds}\Big].
\end{array}
\end{equation}
From (\ref{u10}), (\ref{u11}) and (\ref{u6}) we have the following two conditions:
\begin{equation}\label{u12}
\begin{array}{ll}
\dfrac{1}{\kappa_3}\Big(\dfrac{\kappa_1}{\kappa_2}\int\kappa_1ds\Big)^{\prime}+\dfrac{\kappa_2}{\kappa_3}
=&\Big(A-\int\Big[\kappa_2(s)\sin{\int\kappa_3(s)ds}\Big]ds\Big)\cos{\int\kappa_3(s)ds}\\
&+
\Big(B+\int\Big[\kappa_2(s)\cos{\int\kappa_3(s)ds}\Big]ds\Big)\sin{\int\kappa_3(s)ds}.
\end{array}
\end{equation}
and
\begin{equation}\label{u13}
\begin{array}{ll}
\dfrac{\kappa_1}{\kappa_2}\int\kappa_1ds=&\Big(A-\int\Big[\kappa_2(s)\sin{\int\kappa_3(s)ds}\Big]ds\Big)\sin{\int\kappa_3(s)ds}\\
&-\Big(B+\int\Big[\kappa_2(s)\cos{\int\kappa_3(s)ds}\Big]ds\Big)\cos{\int\kappa_3(s)ds}.
\end{array}
\end{equation}
The condition (\ref{u13}) can be written as follows:
$$\begin{array}{ll}
\kappa_1(s)\int\kappa_1(s)ds=&\Big(A-\int\Big[\kappa_2(s)\sin{\int\kappa_3(s)ds}\Big]ds\Big)\kappa_2(s)\sin{\int\kappa_3(s)ds}\\
&-\Big(B+\int\Big[\kappa_2(s)\cos{\int\kappa_3(s)ds}\Big]ds\Big)\kappa_2(s)\cos{\int\kappa_3(s)ds}.
\end{array}$$
If we integrate the above equation we have
\begin{equation}\label{u15}
\begin{array}{ll}
\Big(\int\kappa_1(s)ds\Big)^2=&C-\Big(A-\int\Big[\kappa_2(s)\sin{\int\kappa_3(s)ds}\Big]ds\Big)^2
\\
&-\Big(B+\int\Big[\kappa_2(s)\cos{\int\kappa_3(s)ds}\Big]ds\Big)^2
,
\end{array}
\end{equation}
where $C$ is a constant of integration. From Equations (\ref{u12}) and (\ref{u13}), we get
\begin{equation}\label{u16}
\begin{array}{ll}
\Big[\dfrac{1}{\kappa_3}\Big(\dfrac{\kappa_1}{\kappa_2}\int\kappa_1ds\Big)^{\prime}+\dfrac{\kappa_2}{\kappa_3}\Big]^2+
\Big(\dfrac{\kappa_1}{\kappa_2}\int\kappa_1ds\Big)^2
&=\Big(A-\int\Big[\kappa_2(s)\sin{\int\kappa_3(s)ds}\Big]ds\Big)^2\\
&+
\Big(B+\int\Big[\kappa_2(s)\cos{\int\kappa_3(s)ds}\Big]ds\Big)^2.
\end{array}
\end{equation}
Now Equations (\ref{u15}) and (\ref{u16}) give
\begin{equation}\label{u17}
\Big(\int\kappa_1(s)ds\Big)^2
+\Big[\dfrac{1}{\kappa_3}\Big(\dfrac{\kappa_1}{\kappa_2}\int\kappa_1ds\Big)^{\prime}+\dfrac{\kappa_2}{\kappa_3}\Big]^2
+
\Big(\dfrac{\kappa_1}{\kappa_2}\int\kappa_1ds\Big)^2
=C.
\end{equation}
Moreover this constant $C$ calculates as follows. From (\ref{u17}), together the three equations (\ref{u6}) we have
$$C=\frac{a_1^2+a_3^2+a_4^2}{a_2^2}=\frac{1-a_2^2}{a_2^2}=\tan^2\theta,$$
where we have used  (\ref{u3}) and the fact that $U$ is a unit vector field.

We do the converse of the proof. Assume that the condition (\ref{u17}) is satisfied for a curve $\alpha$. Let $\theta\in\r$ be so that $C=\tan^2\theta$. Define the unit vector $U$  by
\begin{equation}\label{u18}
U=\cos\theta\Bigg[\int\kappa_1ds\,\t+
\n+\dfrac{\kappa_1}{\kappa_2}\int\kappa_1ds\,\b_1
+\Big[\dfrac{1}{\kappa_3}\Big(\dfrac{\kappa_1}{\kappa_2}\int\kappa_1ds\Big)^{\prime}
+\dfrac{\kappa_2}{\kappa_3}\Big]\b_2\Bigg].
\end{equation}
By taking account  (\ref{u17}), a differentiation of $U$ gives that $\dfrac{dU}{ds}=0$, which it means that $U$ is a constant vector. On the other hand, the scalar product between  the unit principal normal vector field $\n$ with $U$ is
$$\langle\n(s),U\rangle=\cos\theta.$$
Thus $\alpha$ is a slant curve. This finishes with the proof of Theorem \ref{th-main}.

%%%%%%%%%%%%%%%%%%%%%%%%%%%%%%%%%%%%%%%%%%%%%%%%%%%%%%%%
\section{Further characterizations of slant helices}
%%%%%%%%%%%%%%%%%%%%%%%%%%%%%%%%%%%%%%%%%%%%%%%%%%%%%%%

In this section we present two new characterizations of slant helices. The first one is a consequence of Theorem \ref{th-main}.

\begin{theorem} \label{th-2}Let $\alpha:I\subset R\rightarrow \e^4$ be a unit speed curve in Euclidean space $\e^4$. Then $\alpha$ is a slant helix if and only if there exists a $C^2$-function $f$ such that
\begin{equation}\label{u24}
\kappa_3\,f(s)=\Big(\dfrac{\kappa_1}{\kappa_2}\int\kappa_1ds\Big)^{\prime}
+\kappa_2,\hspace*{1cm}\dfrac{d}{ds}f(s)=-\dfrac{\kappa_3\kappa_1}{\kappa_2}\int\kappa_1ds.
\end{equation}
\end{theorem}

\begin{proof}
Let now assume that $\alpha$ is a slant helix. A differentiation of (\ref{u17}) gives
\begin{equation}\label{u19}
\begin{array}{ll}
&\Big(\int\kappa_1(s)ds\Big)\Big(\int\kappa_1(s)ds\Big)^{\prime}
+\Big(\dfrac{\kappa_1}{\kappa_2}\int\kappa_1ds\Big)\Big(\dfrac{\kappa_1}{\kappa_2}\int\kappa_1ds\Big)^{\prime}\\
&+\Big[\dfrac{1}{\kappa_3}\Big(\dfrac{\kappa_1}{\kappa_2}\int\kappa_1ds\Big)^{\prime}+\dfrac{\kappa_2}{\kappa_3}\Big]
\Big[\dfrac{1}{\kappa_3}\Big(\dfrac{\kappa_1}{\kappa_2}\int\kappa_1ds\Big)^{\prime}+\dfrac{\kappa_2}{\kappa_3}\Big]^{\prime}
=0.
\end{array}
\end{equation}
After some manipulations,  the equation (\ref{u19}) takes the following form
\begin{equation}\label{u20}
\dfrac{\kappa_1\kappa_3}{\kappa_2}\int\kappa_1ds+\Big[\dfrac{1}{\kappa_3}\Big(\dfrac{\kappa_1}{\kappa_2}\int\kappa_1ds\Big)^{\prime}
+\dfrac{\kappa_2}{\kappa_3}\Big]^{\prime}=0.
\end{equation}
If we define $f=f(s)$ by
$$\kappa_3\,f(s)=\Big(\dfrac{\kappa_1}{\kappa_2}\int\kappa_1ds\Big)^{\prime}
+\kappa_2.$$
Then Equation (\ref{u20}) writes as
$$\dfrac{d}{ds}f(s)=-\dfrac{\kappa_3\kappa_1}{\kappa_2}\int\kappa_1 ds.$$
Conversely, if (\ref{u24}) holds, we define a unit constant vector $U$ by
$$U=\cos\theta\Bigg[\int\kappa_1ds\,\t+
\n+\dfrac{\kappa_1}{\kappa_2}\int\kappa_1ds\,\b_1
+f(s)\b_2\Bigg].$$
We have that $\langle\n(s),U\rangle=\cos\theta$ is constant, that is, $\alpha$ is a slant helix.
\end{proof}

We end  giving an integral characterization of a slant helix.

\begin{theorem}\label{th-33} Let  $\alpha:I\subset R\rightarrow \e^4$ be a unit speed curve in Euclidean space $\e^4$. Then $\alpha$ is a slant helix if and only if the following condition is satisfied
\begin{equation}\label{u244}
\begin{array}{ll}
\dfrac{\kappa_1}{\kappa_2}\int\kappa_1ds=&\Big(A-\int\Big[\kappa_2(s)\sin{\int\kappa_3(s)ds}\Big]ds\Big)\sin{\int\kappa_3(s)ds}\\
&-\Big(B+\int\Big[\kappa_2(s)\cos{\int\kappa_3(s)ds}\Big]ds\Big)\cos{\int\kappa_3(s)ds},
\end{array}
\end{equation}
for some constants $A$ and $B$.
\end{theorem}
\begin{proof}
Suppose that $\alpha$ is a slant helix. By using  Theorem \ref{th-2}, let define $m(s)$ and $n(s)$ by
\begin{equation}\label{u25}
\phi=\phi(s)=\int^s\kappa_3(u)du,
\end{equation}
\begin{equation}\label{u26}
\begin{array}{ll}
m(s)=&f(s)\cos\phi+\Big(\dfrac{\kappa_1}{\kappa_2}\int\kappa_1ds\Big)\sin\phi+\int\Big[\kappa_2\sin\phi\Big]ds,\\
n(s)=&f(s)\sin\phi-\Big(\dfrac{\kappa_1}{\kappa_2}\int\kappa_1ds\Big)\cos\phi-\int\Big[\kappa_2\cos\phi\Big]ds.
\end{array}
\end{equation}
If we differentiate Equations (\ref{u26}) with respect to $s$ and taking into  account of (\ref{u25}) and (\ref{u24}), we
obtain  $\dfrac{dm}{ds}=0$ and $\dfrac{dn}{ds}=0$. Therefore, there exist constants $A$ and $B$ such that $m(s)=A$ and $n(s)=B$.
By substituting  into (\ref{u26}) and solving the resulting equations for $\dfrac{\kappa_1}{\kappa_2}\int\kappa_1ds$, we get
$$
\dfrac{\kappa_1}{\kappa_2}\int\kappa_1ds=\Big(A-\int\Big[\kappa_2(s)\sin{\phi}\Big]ds\Big)\sin{\phi}
-\Big(B+\int\Big[\kappa_2(s)\cos{\phi}\Big]ds\Big)\cos{\phi}.
$$
Conversely, suppose that (\ref{u244}) holds. In order to apply Theorem \ref{th-2}, we define $f=f(s)$ by
\begin{equation}\label{u28}
f(s)=\Big(A-\int\Big[\kappa_2(s)\sin{\phi}\Big]ds\Big)\cos{\phi}\\
+
\Big(B+\int\Big[\kappa_2(s)\cos{\phi}\Big]ds\Big)\sin{\phi},
\end{equation}
with $\phi(s)=\int\kappa_3(u) du$. A direct differentiation of (\ref{u244}) gives
$$\Bigg(\dfrac{\kappa_1}{\kappa_2}\int\kappa_1ds\Bigg)'=\kappa_3 f(s)-\kappa_2.$$
This shows the left condition in (\ref{u24}). Moreover, a straightforward computation leads to
$f'(s)=-\dfrac{\kappa_3\kappa_1}{\kappa_2}\int\kappa_1ds$, which finishes the proof.
\end{proof}

%%%%%%%%%%%%%%%%%%%%%%%%%%%%%%%%%%%%%%%

\end{document}